\documentclass[reqno]{amsart}
\usepackage{stmaryrd}
\usepackage{silence}
\usepackage{upgreek}
\usepackage{amssymb}
\usepackage{faktor}
\usepackage[all,cmtip]{xy}
\usepackage{amsmath} 
\usepackage{mathrsfs}
\usepackage{tikz}
\usepackage{tikz-cd}
\usepackage{enumitem}
\usepackage{mathtools}
\usepackage[mathcal]{euscript}
\usepackage{graphicx}
\usepackage{url}
\usepackage{hyperref}
\usepackage[T1]{fontenc}

\usetikzlibrary{shapes.geometric}
\usetikzlibrary{decorations.markings}
\usetikzlibrary{patterns,decorations.pathreplacing}

\hypersetup{colorlinks=true,linkcolor=blue}

\usepackage{ragged2e}

\definecolor{coloryellow}{RGB}{240,228,66}
\definecolor{colorskyblue}{RGB}{86,180,233}
\definecolor{colorvermillion}{RGB}{213,94,0}

\newcommand{\graphfont}{\mathsf}

\newcommand{\graf}{\graphfont{\Gamma}}

\DeclareSymbolFont{sfletters}{OT1}{cmss}{m}{n}
\DeclareMathSymbol{\sTheta}{\mathord}{sfletters}{"02}

\theoremstyle{definition}
\newtheorem{definition}{Definition}[section]

\theoremstyle{plain}
\newtheorem{proposition}[definition]{Proposition}
\newtheorem{lemma}[definition]{Lemma}

\theoremstyle{remark}
\newtheorem{remark}[definition]{Remark}

\makeatletter
\@addtoreset{definition}{section}
\makeatother

\usepackage{marginnote}
    \DeclareFontFamily{U}{wncy}{}
    \DeclareFontShape{U}{wncy}{m}{n}{<->wncyr10}{}
    \DeclareSymbolFont{mcy}{U}{wncy}{m}{n}
    \DeclareMathSymbol{\Sha}{\mathord}{mcy}{"58}

\newsavebox{\foobox}

\title{Corrigendum to ``Asymptotic homology of graph braid groups''}
\author{Byung Hee An}
\author{Gabriel C. Drummond-Cole}
\author{Ben Knudsen}
\begin{document}
\begin{abstract}
We correct an error in the paper referred to in the title. Part of the argument is organized as a general method for establishing when (derived) functors factor through a fixed Serre subcategory, which may be of some more general interest.
\end{abstract}
\maketitle

\section{Introduction}

The purpose of this note is to correct the proof of \cite[Lem. 3.15]{AnDrummond-ColeKnudsen:AHGBG}. We begin by recalling some of the relevant concepts and notation. First, $\graf$ denotes a graph with set of edges $E(\graf)$, and $\graf_v$ denotes the graph obtained by adding a bivalent vertex $v$ to some (fixed) edge $e$ of $\graf$, then exploding this vertex---see \cite[\S2.2]{AnDrummond-ColeKnudsen:AHGBG}. The set of edges of $\graf_v$ is the same as that of $\graf$, except that the fixed edge of $\graf$ has been split into two edges, denoted $e$ and $e'$.

Second, we will be interested in (unordered) partitions $P$ of a given background set $X$, which is to say $P$ is a set of disjoint subsets, called blocks, whose union covers $X$. We write $x\sim_P y$ when $x$ and $y$ lie in the same block of $P$. In our setting, $X$ will be a set of generators of the polynomial ring $\mathbb{Z}[X]$, which we regard as a graded ring with $X$ in weight $1$. The partition $P$ determines a quotient map $\mathbb{Z}[X]\to \mathbb{Z}[P]$, by which we view the target as a module over the source.

When the background set is the set of edges of a graph, we may speak of ``tame'' partitions. The details of this concept are mostly irrelevant here, but the interested reader may consult \cite[Def. 3.2 ]{AnDrummond-ColeKnudsen:AHGBG}.\footnote{Such a reader will notice in this definition the presence of a parameter corresponding to the homological degree. For us, this parameter will always be fixed, so we suppress it for ease of language.} The modules determined by all tame partitions of $E(\graf)$, together with their graded shfits, generate a Serre subcategory, the category of tame $\mathbb{Z}[E(\graf)]$-modules. Writing $\mathrm{Tor}_{e-e'}(M)$ for the $(e-e')$-torsion submodule of $M$, our main result is the following.

\begin{proposition}\label{prop:main}
If $M$ is a tame $\mathbb{Z}[E(\graf_v)]$-module, then the $\mathbb{Z}[E(\graf)]$-modules $M/(e-e')$ and $\mathrm{Tor}_{e-e'}(M)$ are also tame.
\end{proposition}

The first claim of this result is the lemma of interest from \emph{loc. cit.} The second claim was mistakenly asserted there to be obvious and was used to deduce the first; in fact, the two are both logically equivalent and of equivalent difficulty. 

Briefly, our strategy will be to observe that both claims hold for the generators of the category of tame modules, then to argue by induction on a natural filtration of that category. It is worth emphasizing that Noetherianity, which made no appearance in the erroneous argument, is used in an essential way here.

In more detail, the argument is divided into three parts, which correspond to the three short sections following the introduction. The first part, valid in great generality, develops a framework for establishing the factorization of a right exact functor, and its derived functors, through a Serre subcategory of the target. The second part, still somewhat general, applies this framework to the setting of modules over a Noetherian ring. The third part specializes to the setting of \emph{loc. cit.} to prove the result in question using what has come before.

\subsection{Conventions} Gradings of rings and modules are by ``weight'' rather than ``degree,'' which is to say that we use the symmetry isomorphism of the tensor product incorporating no Koszul signs. Thus, commutative graded rings are strictly commutative rather than ``graded commutative.''

\subsection{Acknowledgements} The third author benefited from the support of NSF grant DMS 2551600.

\section{Factoring through a Serre subcategory}

In this section, we fix a right exact functor $F:\mathcal{A}\to \mathcal{B}$ between Abelian categories, where $\mathcal{A}$ has enough projectives, together with a subcategory $\mathcal{B}_0\subseteq\mathcal{B}$. Denoting the $i$th left derived functor of $F$ by $F_i$, we will be interested in understanding when $F_i$ factors through $\mathcal{B}_0$, perhaps after restriction to a subcategory of its domain.

\begin{definition}
With respect to a subcategory $\mathcal{C}\subseteq \mathcal{A}$, we say that $F$ has property
\begin{itemize}
\item[{\bf O(i)}] if $F_i(A)\in\mathcal{B}_0$ whenever $A$ is an {\bf{object}} of $\mathcal{C}$;
\item[{\bf Q(i)}] if $F_i(A)\in\mathcal{B}_0$ whenever $A$ is a {\bf{quotient}} of an object of $\mathcal{C}$;
\item[{\bf S(i)}] if $F_i(A)\in\mathcal{B}_0$ whenever $A$ is a {\bf{subobject}} of an object of $\mathcal{C}$; and
\item[{\bf E(i)}] if $F_i(A)\in\mathcal{B}_0$ whenever $A$ is an {\bf{extension}} of objects of $\mathcal{C}$.
\end{itemize}
By convention, each condition is vacuous for $i<0$.
\end{definition}

We record the following simple observation concerning these properties.

\begin{lemma}\label{lem:property lemma}
Suppose that $\mathcal{B}_0$ is a Serre subcategory. 
\begin{enumerate}
\item If property {\bf O(i)} holds with respect to $\mathcal{C}$, then so does property {\bf{E(i)}}.
\item If properties {\bf O(i)} and {\bf S(i-1)} hold with respect to $\mathcal{C}$, then so does property {\bf Q(i)}.
\end{enumerate}
\end{lemma}
\begin{proof}
We begin with an exact sequence $0\to A\to B\to C\to 0$. For the first claim, consider the resulting exact sequence $F_i(A)\to F_i(B)\to F_i(C)$. By {\bf O(i)}, if $A$ and $C$ lie in $\mathcal{C}$, then $F_i(A)$ and $F_i(C)$ lie in $\mathcal{B}_0$, and the claim follows by our assumption that $\mathcal{B}_0$ is Serre. For the second claim, consider the resulting exact sequence $F_i(B)\to F_i(C)\to F_{i-1}(A)$. By {\bf O(i)} and {\bf S(i-1)}, if $B$ lies in $\mathcal{C}$, then $F_i(B)$ and $F_{i-1}(A)$, respectively, lie in $\mathcal{B}_0$ and the claim follows as before.
\end{proof}

To state the main result of this section, we fix a subcategory $\mathcal{G}\subseteq \mathcal{A}$ and write $\mathcal{A}_0\subseteq \mathcal{A}$ for the Serre subcategory it generates.

\begin{proposition}\label{prop:properties}
Suppose that $\mathcal{B}_0$ is a Serre subcategory. If properties {\bf O(i)} and {\bf S(i)} hold with respect to $\mathcal{G}$ for every $i\geq0$, then property {\bf O(i)} holds with respect to $\mathcal{A}_0$ for every $i\geq0$.
\end{proposition}

\begin{remark}\label{remark:implies}
Since $\mathcal{A}_0$ is Serre, it follows in the situation of Proposition \ref{prop:properties} that the other three properties also hold for every $i\geq0$.
\end{remark}

For the proof, we will make use of a sequence of categories interpolating between $\mathcal{G}$ and $\mathcal{A}_0$.

\begin{definition}\label{def:type}
We say that $A\in\mathcal{A}$ is of \emph{type $0$} if $A\in\mathcal{G}$. Recursively, we say that $A$ is of type $n$ if it is a subobject, quotient, or extension of objects of type $n-1$.
\end{definition}

Writing $\mathcal{A}_0^{(n)}\subseteq \mathcal{A}$ for the full subcategory spanned by all objects of type $n$, we have $\mathcal{G}=\mathcal{A}_0^{(0)}$ and $\mathcal{A}_0^{(n)}\subseteq\mathcal{A}_0^{(n+1)}$. By convention, we set $\mathcal{A}_0^{(n)}=\{0\}$ for $n<0$.

\begin{lemma}\label{lem:union}
We have $\mathcal{A}_0=\bigcup_{n\geq0}\mathcal{A}_0^{(n)}$.
\end{lemma}
\begin{proof}
Since $\mathcal{A}_0$ is a Serre subcategory containing $\mathcal{G}$, it follows by induction that $\mathcal{A}_0$ contains $\mathcal{A}_0^{(n)}$ for every $n\geq0$; therefore, by the definition of $\mathcal{A}_0$, it suffices to show that the union is a Serre subcategory, but this is immediate from Definition \ref{def:type}.
\end{proof}

The main technical step in the argument is the following.

\begin{lemma}\label{lem:other induction}
If properties {\bf Q(i)}, and {\bf E(i)} hold with respect to $\mathcal{A}_0^{(n+1)}$, and if property {\bf S(i)} holds with respect to $\mathcal{A}_0^{(n)}$, then property {\bf S(i)} holds with respect to $\mathcal{A}_0^{(n+1)}$.
\end{lemma}
\begin{proof}
Given a subobject $S\subseteq A$ with $A$ of type $n+1$, we wish to show that $F_i(S)\in \mathcal{B}_0$. There are three cases to consider. First, if $A$ is a subobject of the type $n$ object $B$, then $S$ is also a subobject of $B$; therefore, since property {\bf S(i)} holds with respect to $\mathcal{A}_0^{(n)}$, the desired conclusion follows. Second, suppose that $A$ is a quotient of the type $n$ object $B$, and consider the pullback diagram
\[\xymatrix{
P\ar[r]\ar[d]& B\ar[d]\\
S\ar[r]&A.
}\] Since the lower arrow is a monomorphism, the upper arrow is so as well, so $P\in\mathcal{A}_0^{(n+1)}$. Since the righthand arrow is an epimorphism, the lefthand arrow is so as well (we use that $\mathcal{A}$ is Abelian), so the desired conclusion follows from property {\bf Q(i)} with respect to $\mathcal{A}_0^{(n+1)}$. Finally, suppose given the exact sequence 
$0\to B\to A\to C\to 0$ with $B$ and $C$ of type $n$. We enlarge this sequence to the commutative diagram with exact rows
\[\xymatrix{
0\ar[r]&P\ar[d]\ar[r]&S\ar[d]\ar[r]&S/P\ar[r]\ar[d]&0\\
0\ar[r]&B\ar[r]&A\ar[r]& C\ar[r]&0,
}\] where $P$ is defined by requiring the lefthand square to be a pullback. By the first isomorphism theorem, the map $S/P\to C$ is isomorphic to the inclusion of the image of $S\to C$; in particular, it is a monomorphism, and we conclude that $P$ and $S/P$ are both of type $n+1$. Since property {\bf E(i)} holds with respect to $\mathcal{A}_0^{(n+1)}$, the claim follows.
\end{proof}

We also record the following lemma, which is immediate from the definitions.

\begin{lemma}\label{lem:induction}
If properties {\bf O(i)}, {\bf Q(i)}, {\bf S(i)}, and {\bf E(i)} hold with respect to $\mathcal{A}_0^{(n)}$, then property {\bf O(i)} holds with respect to $\mathcal{A}_0^{(n+1)}$.
\end{lemma}

\begin{proof}[Proof of Proposition \ref{prop:properties}]
Let $P(i,n)$ be the statement that properties {\bf O(i)}, {\bf Q(i)}, {\bf S(i)}, and {\bf E(i)} hold with respect to $\mathcal{A}_0^{(n)}$. By Lemma \ref{lem:union}, the claim is equivalent to the validity of $P(i,n)$ for every $i,n\geq0$ (see Remark \ref{remark:implies}).

We begin by observing that the statements $P(i,n-1)$ and $P(i-1,n)$ together imply $P(i,n)$. Assuming $P(i,n-1)$, Lemma \ref{lem:induction} implies that property {\bf O(i)} holds with respect to $\mathcal{A}^{(n)}_0$, hence property {\bf E(i)} by Lemma \ref{lem:property lemma}(1). Assuming $P(i-1,n)$ as well, Lemma \ref{lem:property lemma}(2) grants that property {\bf Q(i)} holds with respect to $\mathcal{A}^{(n)}_0$. Finally, $P(i,n-1)$ and Lemma \ref{lem:other induction} now imply $P(i,n)$.

Since $P(i,0)$ holds for every $i\geq0$ by our assumption and Lemma \ref{lem:property lemma}, it therefore suffices to show that $P(0,n)$ holds for every $n\geq0$. Since $P(-1,n)$ is vacuous, the argument above shows that $P(0,n-1)$ implies $P(0,n)$, so induction on $n$ reduces the claim to $P(0,0)$, which holds by assumption.
\end{proof}

\section{The role of Noetherianity}

We now specialize the setup of the previous section by taking $\mathcal{A}$ and $\mathcal{B}$ to be the categories of finitely generated graded modules over the commutative graded rings $R$ and $S$, respectively. We fix a set of generators $\mathcal{G}\subseteq \mathcal{A}$, a right exact functor $F:\mathcal{A}\to\mathcal{B}$, and a Serre subcategory $\mathcal{B}_0\subseteq\mathcal{B}$.

\begin{proposition}\label{prop:cyclic}
Suppose that $R$ is Noetherian. If $\mathcal{G}$ is closed under the formation of cyclic submodules, then property {\bf S(0)} holds with respect to $\mathcal{G}$ provided property {\bf O(0)} does so.
\end{proposition}
\begin{proof}
Let $S$ be a submodule of the generator $G\in\mathcal{G}$. By Noetherianity, there exist $x_i\in S$ for $1\leq i\leq n$ such that 
\[S=\bigcup_{i=1}^nRx_i,\] where $Rx_i\subseteq G$ is the cyclic submodule generated by $x_i$. By right exactness, then, the canonical map
\[\bigoplus_{i=1}^n F_0(Rx_i)\to F_0(S)\] is an epimorphism. Our assumption on $\mathcal{G}$ and property {\bf O(0)} guarantee that each summand of the source lies in $\mathcal{B}_0$, and, since $\mathcal{B}_0$ is Serre, it follows that the target does as well.
\end{proof}

\section{Application to tame modules}

We now specialize further to the setting of \cite{AnDrummond-ColeKnudsen:AHGBG}, as partially recalled in the introduction. To begin, we record the following, which is Lemma 3.14 of \emph{loc. cit.}

\begin{lemma}\label{lem:quotient partition}
If $P$ is tame with respect to $\graf_v$, then the partition obtained by merging the block of $P$ containing $e$ with the block containing $e'$ is tame with respect to $\graf$.
\end{lemma}

We also require the following simple observation.

\begin{lemma}\label{lem:cyclic submodules}
Let $P$ be a partition of the set $X$. Every cyclic $\mathbb{Z}[X]$-submodule of $\mathbb{Z}[P]$ is isomorphic to a graded shift of $\mathbb{Z}[P]$.
\end{lemma}
\begin{proof}
It suffices to consider the case $X=P$, where the claim is equivalent to the fact that $\mathbb{Z}[X]$ is an integral domain.
\end{proof}

In order to apply the general theory developed above, we now identify terms. The Abelian category $\mathcal{A}$ is the category of finitely generated graded $\mathbb{Z}[E(\graf_v)]$-modules (resp. $\mathcal{B}$, $\graf$). The Serre subcategory $\mathcal{A}_0$ is generated by the collection $\mathcal{G}$ of graded shifts of modules of the form $\mathbb{Z}[P]$ with $P$ a tame partition of $E(\graf_v)$, and $\mathcal{B}_0$ is generated by the corresponding collection for $\graf$. The functor $F$ is the change of scalars functor along the natural homomorphism $\mathbb{Z}[E(\graf_v)]\to \mathbb{Z}[E(\graf)]$, i.e., $F(M)=M/(e-e')$.

With these specifications, Proposition \ref{prop:main} is that properties {\bf O(0)} and {\bf O(1)} hold with respect to $\mathcal{A}_0$.

\begin{proof}[Proof of Proposition \ref{prop:main}]
By Proposition \ref{prop:properties}, and since $F_i\equiv0$ for $i>1$, it suffices to establish that properties {\bf O(i)} and {\bf S(i)} hold with respect to $\mathcal{G}$ for $i\in\{0,1\}$. Property {\bf O(0)} is precisely Lemma \ref{lem:quotient partition}, and Proposition \ref{prop:cyclic} and Lemma \ref{lem:cyclic submodules} establish property {\bf S(0)}. As for {\bf O(1)}, evidently we have 
\[F_1(\mathbb{Z}[P])\cong\begin{cases}
\mathbb{Z}[P] &\quad \text{if }e\sim_P e'\\
0 &\quad \text{otherwise},
\end{cases}\] which lies in $\mathcal{B}_0$ by Lemma \ref{lem:quotient partition}. Finally, property {\bf S(1)} follows from property {\bf O(1)} and the fact that $F_1$ is left exact.
\end{proof}

\bibliographystyle{amsalpha}
\bibliography{references}

\providecommand{\bysame}{\leavevmode\hbox to3em{\hrulefill}\thinspace}
\providecommand{\MR}{\relax\ifhmode\unskip\space\fi MR }
\providecommand{\MRhref}[2]{%
  \href{http://www.ams.org/mathscinet-getitem?mr=#1}{#2}
}
\providecommand{\href}[2]{#2}
\begin{thebibliography}{ADCK22}

\bibitem[ADCK22]{AnDrummond-ColeKnudsen:AHGBG}
B.~H. An, G.~C. Drummond-Cole, and B.~Knudsen, \emph{Asymptotic homology of
  graph braid groups}, Geom. Topol. \textbf{26} (2022), 1745--1771.

\end{thebibliography}

\end{document}